\documentclass[a4paper,12pt]{amsart}
\usepackage{amsfonts,amsthm,amssymb,amsmath,amscd}
\usepackage[colorlinks=true,linkcolor=blue,citecolor=blue]{hyperref}
\usepackage{amsmath}
\usepackage{tikz-cd}

\usepackage{tabularx}
\usepackage{tcolorbox}
\usepackage[english]{babel}
\usepackage{amssymb,amsmath}
\usepackage{xcolor}

\newtheorem{Theo}{Theorem}[section]
\newtheorem{Prop}[Theo]{Proposition}
\newtheorem{Coro}[Theo]{Corollary}
\newtheorem{Lemm}[Theo]{Lemma}

\newcommand{\T}{\mathbb{T}}
\newcommand{\Bcal}{\mathcal{B}}

\newcommand{\C}{\mathbb{C}}

\def\N{\mathbb{ N}}
\def\R{\mathbb{ R}}

\begin{document}

\title{Henry Helson meets other big shots -- a brief survey}


\author[Defant]{Andreas Defant}
\address[]{Andreas Defant\newline  Institut f\"{u}r Mathematik,\newline Carl von Ossietzky Universit\"at,\newline
26111 Oldenburg, Germany.
}
\email{defant@mathematik.uni-oldenburg.de}

\author[Schoolmann]{Ingo Schoolmann}
\address[]{Ingo Schoolmann\newline  Institut f\"{u}r Mathematik,\newline Carl von Ossietzky Universit\"at,\newline
26111 Oldenburg, Germany.
}
\email{ingo.schoolmann@uni-oldenburg.de}

\maketitle

\begin{abstract}
\noindent
A theorem of Henry Helson shows that for every ordinary Dirichlet series $\sum a_n n^{-s}$ with a square summable sequence $(a_n)$ of coefficients, almost all vertical limits
$\sum a_n \chi(n) n^{-s}$, where $\chi: \mathbb{N} \to \mathbb{T}$ is a completely multiplicative
 arithmetic function, converge on the right half-plane. We survey on recent improvements and extensions of this result within
 Hardy spaces of Dirichlet series -- relating it with some classical work of Bohr, Banach, Carleson-Hunt, Ces\`{a}ro,
  Hardy-Littlewood, Hardy-Riesz,  Menchoff-Rademacher, and Riemann.
\end{abstract}


\noindent
\renewcommand{\thefootnote}{\fnsymbol{footnote}}
 \footnotetext{}

\section{Introduction}
In his article  \cite{Helson1} from 1967 Henry Helson suggested that the classical theory of ordinary Dirichlet series
should be combined with modern harmonic analysis and functional analysis, and since then this theory saw a remarkable comeback -- in particular after the publication of the seminal articel \cite{HLS}.

The aim of this survey is to discuss various recent variants of a somewhat curious theorem of Helson dealing with  Dirichlet series $D=\sum a_n n^{-s}$ which have  $2$-summable coefficients. These Dirichlet series form the natural
 Hilbert space $\mathcal{H}_2$, and to see a very first example
look for $\varepsilon>0$ at the following translated zeta series
\begin{equation} \label{zeta}
D = \sum \frac{1}{n^{\frac{1}{2}+\varepsilon}} n^{-s}\,.
\end{equation}
In general, each  $D\in \mathcal{H}_2$  converges for all $s \in \mathbb{C}$ in the half-plane $[Re > 1/2]$. But if we multiply
the $a_n$'s with some character $\chi$, i.e.  a
completely multiplicative arithmetic functions $\chi: \mathbb{N} \to \mathbb{T}$, and  consider the new Dirichlet series
     $D^\chi =\sum a_{n} \chi(n) n^{-s}$, then the convergence in general improves considerably.
     We refer to the following remarkable theorem from \cite[Theorem, p. 140]{Helson3} as 'Helson's theorem'.

    \begin{Theo} \label{Helson}
   Let $D=\sum a_{n}n^{-s}$ be a Dirichlet series with coefficients $(a_n) \in \ell_2$.
   Then for almost all characters  $\chi: \mathbb{N} \to \mathbb{T}$
   we have that
     $D^\chi=\sum a_{n} \chi(n) n^{-s}$  converges on
     all of $[\text{Re} >0]$.
   \end{Theo}

What is here meant by 'almost all characters $\chi$'? A simple way to understand this is to identify
the set $\Xi$ of all  characters $\chi\colon \N \to \T$
with the  infinitely dimensional torus $\T^{\infty}$, i.e. the  countable product of $\mathbb{T} = \{w \in \mathbb{C}\colon |w|=1\}$ which forms a natural compact abelian group, where
  the Haar measure  is given by the  normalized Lebesgue measure $dz$.  We write $\mathfrak{p} = 2,3,5, \ldots$ for the  sequence of prime numbers. If we consider
   pointwise multiplication on  $\Xi$, then
\begin{align} \label{idy}
\iota\colon \Xi \to \T^{\infty}, ~~ \chi \mapsto \chi(\mathfrak{p}) = (\chi(p_{n}))_{n},
\end{align}
 is a  group isomorphism which turns  $\Xi$ into a compact abelian group. The Haar measure $d \chi$ is the push forward measure of $dz$  through $\iota^{-1}$.

  Applying his theorem to the Dirichlet series from \eqref{zeta}, Helson detects as a somewhat curious  application   that   'Riemann's conjecture holds true almost everywhere' in the following sense.

  \begin{Theo} \label{riemann}
  For almost all $\chi \in \Xi$ the Dirichlet series  $\zeta^\chi = \sum \chi(n) n^{-s}$
  of the Riemann zeta series $\zeta =\sum n^{-s}$ have no zeros in the critical half-plane  $[\text{Re}>1/2]$.
  \end{Theo}

  Let us come back to the Hilbert space $\mathcal{H}_2$, and look at it from a view point originally invented by H. Bohr in \cite{BohrStrip}. To understand this recall first that the set of all characters on $\mathbb{T}^\infty$, so the dual group of $\mathbb{T}^\infty$, consists  of all  monomials $z \mapsto z^\alpha$, where
 $\alpha = (\alpha_k)\in \mathbb{Z}^{(\mathbb{N})}$ (all  finite sequences
 of integers).
    Obviously $D = \sum a_n n^{-s} \in \mathcal{H}_2 $ if and only if there  is  a (then unique) function
  $f \in H_2(\mathbb{T}^\infty)$ such that $a_n = \hat{f}(\alpha)$ for all $\alpha \in \mathbb{N}_0^{(\mathbb{N})}$ with  $\mathfrak{p}^\alpha = n$. In other terms, the mapping
  \[
  H_2(\mathbb{T}^\infty) \to \mathcal{H}_2, \,\,\,f \mapsto D\,,
 \]
identifies two Hilbert spaces.

  Now the following two questions appear naturally. If  $D\in \mathcal{H}_2$ satisfies the assertion from  Helson's theorem,
 what does this  mean for the associated function  $f \in H_2(\mathbb{T}^\infty)$? And vice versa, if we have an
 appropriate  theorem on pointwise  convergence for functions in $H_2(\mathbb{T}^\infty)$, when does it transfer to a Helson-like theorem for ordinary Dirichlet?

 Let us give four examples which indicate that it is worth to look at such an interplay more carefully.
 All four examples come along with some more precise questions.\\

 \noindent
{\bf Ex 1:}
\label{2}
Consider for $f \in H_2(\mathbb{T}^\infty)$ and $u>0$    the 'translated' orthonormal series
$$\sum \hat{f}(\alpha)\mathfrak{p}^{-u\alpha} z^\alpha\,,$$
 which clearly   defines a function $f_u \in H_2(\mathbb{T}^\infty)$. Then an  immediate translation of   Helson's theorem through \eqref{idy}
  shows that
  \begin{equation} \label{HelsonA}
  f_u(z)=\lim_{x \to \infty }\sum_{\mathfrak{p}^\alpha < x} \hat{f}(\alpha)\mathfrak{p}^{-u\alpha} z^\alpha
 \,\,\,\, \text{almost everywhere on $\mathbb{T}^\infty$}\,.
 \end{equation}
 Does this result even hold for $u =0$, and if yes, what does this in turn then mean for Helson's theorem?
 Note that  in contrast to \eqref{HelsonA}, only if  $u > 1/2$, for all  $f \in H_2(\mathbb{T}^\infty)$
 \[
  f_u(z)=\sum_{\alpha} \hat{f}(\alpha)\mathfrak{p}^{-u\alpha} z^\alpha
 \,\,\,\, \text{almost everywhere on $\mathbb{T}^\infty$}
 \]
 (here the series only makes sense if we consider absolute convergence); see e.g. \cite[Remark 11.3]{Defant}.\\

  \noindent
{\bf Ex 2:} \label{2}
For each $s \in \mathbb{C}$ we may interpret  $D^\chi=\sum (a_{n}n^{-s}) \chi(n)$ from Helson's theorem
as an orthonormal series in $L_2(\Xi)$.
This brings us to recall the Menchoff-Rademacher theorem, which is  the fundamental theorem on  almost everywhere convergence of general orthonormal series, and as most convergence theorems it is accompanied by an maximal inequality which was isolated by Kantorovitch;  see e.g. the standard reference \cite{Alexits}.

  \begin{Theo}\label{MenchoffRademacher}
   Let $\sum c_n x_n$ be an orthonormal sequence in some $L_2(\mu)$, i.e. $(c_n) \in \ell_2$
   and $(x_n)$ an orthonormal sequence. Then $\sum c_n x_n$ converges $\mu$-almost everywhere
   whenever $(c_n \log n) \in \ell_2$. Moreover,
   there is a constant $C >0$ such that for all $c_1, \ldots, c_x \in \mathbb{C}$
   \begin{equation} \label{Kantorovitch}
            \Big( \int \sup_{x} \Big| \sum_{n=1}^x c_n x_n \Big|^{2} d\mu \Big)^{\frac{1}{2}}
           \le C \big\|(c_n \log n)\big\|_2\,.
             \end{equation}
        \end{Theo}
        Does the Menchoff-Rademacher theorem imply Helson's theorem, and does Kantorovitch's inequality
        even add its relevant maximal inequality? \\

 \noindent
{\bf Ex 3:} \label{3}
Of course, for special orthonormal sequences $(x_n)$ the assertion of the Menchoff-Rademacher theorem  improves considerably, and the most celebrated  theorem
in this direction is certainly due to Carleson. In this case it was Hunt who after a careful analysis of Carleson's work, came up with what is now known as the Carleson-Hunt maximal inequality (see also Theorem  \ref{CarlesonHunt_p} for $1 <p< \infty$).

\begin{Theo} \label{CarlesonHunt}
   The Fourier series
   $
   f(z) = \sum_{n=0}^\infty \hat{f}(n) z^n
   $
   of every   $f \in H_2(\mathbb{T})$ converges almost everywhere on $\mathbb{T}$. Moreover, there is a (best) constant $\text{CH}_{2} >0$ such that
   \begin{equation*}
\Big( \int_{\mathbb{T}} \sup_{x} \Big| \sum_{n < x} \widehat{f}(n) z^n \Big|^{2} dz \Big)^{\frac{1}{2}} \le \text{CH}_{2}\,\|f\|_{2}.
\end{equation*}
  \end{Theo}

Does this deep result give new input to Helson's theorem?
Indeed, look at functions $f \in H_2(\mathbb{T})$, or equivalently at  functions in $H_2(\mathbb{T}^\infty)$
which only depend on the first variable, and $u=0$.  Then by Carleson's theorem
every  Dirichlet series $\sum a_n  \chi(n) n^{-s}$, which is 'thin' in the sense that  $a_n \neq 0$ only if $n=2^j$ for some $j$,  converges for almost all $\chi \in \Xi$
in $s=0$, and consequently for almost all  $\chi$ also on  the right half-plane.
Does this result extend to all
$D \in \mathcal{H}_2$? Equivalently, does \eqref{HelsonA} hold for $u=0$? Up to which extent is it possible to have some  Carleson's theorem  for functions on the infinite dimensional torus?\\

 \noindent
{\bf Ex 4:}  \label{4}
Finally, we  turn to another fundamental theorem on ordinary Dirichlet series.
For its formulation we recall that the Banach space of all Dirichlet series
 $D= \sum a_{n}n^{-s}\,,$ which converge on the right  half-plane to a bounded (and then necessarily holomorphic) function $f$, is denoted by  $\mathcal{D}_\infty$ (the norm given by the sup norm on the right half-plane ).
 The following result due to Bohr \cite{Bohr} rules the theory of such  series.

  \begin{Theo} \label{Bohr}
   Let $D= \sum a_{n}n^{-s} \in \mathcal{D}_\infty$. Then $D$ for each $u >0$ converges uniformly on
   $[\text{Re}> u]$. Even more, there is a constant $C >0$ such that for every $x>1$ we have
   \[
   \sup_{\text{Re} s >0} \Big| \sum_{n=1}^x a_n n^{-s} \Big| \leq C\,\log x \sup_{\text{Re}  s >0} | f(s) |\,.
   \]
    \end{Theo}
    We intend to explain in which sense this result can be seen as a sort  of extreme case of a scale of Helson-like theorems.\\

We hope to convince our reader that the  above apparently quite different theorems have a lot in common. Inspired through the work of Bayart \cite{Bayart}, Bohr \cite{Bohr}, Duy \cite{Duy},
         Hardy-Riesz \cite{HardyRiesz},
        Hedenmalm-Lindqvist-Seip \cite{HLS},
                   Hedenmalm-Saksman \cite{HedenmalmSaksman}, and Helson \cite{Helson3}, among others,  we want to sketch that they in  fact are  intimately linked.
         And  this picture gets  visible if one looks at the above theorems  within the more general scale  of Hardy spaces $\mathcal{H}_p, 1 \leq p \leq \infty$, of ordinary  Dirichlet series as invented by Bayart  in \cite{Bayart} (see Section \ref{Hardyspace}).

        In particular, we want to discuss several results from  our  ongoing research project on  general (!) Dirichlet series
        $\sum_n a_n e^{-\lambda_n s}$ (see \cite{DefantSchoolmann},  \cite{DefantSchoolmann2}, \cite{DefantSchoolmann3}, and  \cite{Schoolmann}), although the present  survey entirely focuses  on ordinary Dirichlet series
        $\sum_n a_n n^{-s}$ only.
        Indeed, many  of the results which we are going to discuss even hold in the
         much wider  setting of general Dirichlet series and their related Fourier analysis on so-called
         Dirichlet groups.

          Why do we here restrict ourself to the ordinary case? In fact we believe that our topic for ordinary Dirichlet series  is interesting in itself, and more important, we  hope that for a reader who is merely interested in the ordinary case, our presentation is particularly useful since it is not  covered by technical difficulties which   are unavoidable in the much wider framework of general Dirichlet series.

           Our survey has eight  sections:  Helson meets Menchoff-Rademacher,  Riemann,  Carleson-Hunt,  Bohr,  Ces\`{a}ro,  Hardy-Riesz,
          Hardy-Littlewood, and  Banach.

\section{Preliminaries}
For all information on  Dirichlet series we refer to the monographs \cite{Defant}, \cite{HardyRiesz}, \cite{Helson}, or \cite{QQ}. In the following we focus on a few facts of particular interest.

\subsection{Kronecker flow} \label{flow}
   The continuous group homomorphism, the so-called Kronecker flow,
\begin{equation} \label{flow}
\beta: \mathbb{R}  \rightarrow \mathbb{T}^\infty\,, \,\,\,\, t \mapsto (p_k^{-it})_{k=1}^\infty
\end{equation}
has dense range.
 Recall that the dual group $\widehat{\mathbb{T}^\infty}$ equals the group $\mathbb{Z}^{(\mathbb{N})}$ of all finite sequences $\alpha$ with entries from $\mathbb{Z}$  in the sense
 that every such $\alpha$ may be identified with the character $z^\alpha$.
  Then for every character $z^\alpha \in \widehat{\mathbb{T}^\infty}$ we have that  $x =  \log \mathfrak{p}^\alpha$
  is the unique real number for which
  $e^{-i x \pmb{\cdot}} = z^\alpha \circ \beta $.
In other words, $\widehat{\mathbb{T}^\infty}$ and
$\{\log \mathfrak{p}^\alpha\colon \alpha \in \mathbb{Z}^{(\mathbb{N})} \}$
can be identified (as sets), and the natural order of $\mathbb{R}$
transfers to $\widehat{\mathbb{T}^\infty}$:
\begin{equation} \label{order}
\alpha \in \mathbb{Z}^{(\mathbb{N})}  \ge 0 \,\,\,\,\, \text{if}\,\,\,\,\, \log \mathfrak{p}^\alpha \ge 0.
\end{equation}
Moreover, recall for $u>0$ the definition of the Poisson kernel $P_u (t) = \frac{1}{\pi} \frac{u}{t^2 +u^2}$ on $\mathbb{R}$ which has $e^{-u |\pmb{\cdot}|}$ as its Fourier transform.
The push forward measure of $P_u$ under the Kronecker flow $\beta: \mathbb{R} \to \mathbb{T}^\infty$
is denoted  by $p_u$. We have that $\widehat{p_u}(\alpha) = \widehat{P_u}(\log \mathfrak{p}^\alpha) = e^{-u |\log\mathfrak{p}^\alpha |}$ for every $\alpha \in \mathbb{Z}^{(\mathbb{N})} = \widehat{\mathbb{T}^\infty}$.

\subsection{Hardy spaces} \label{Hardyspace}

 Bayart in \cite{Bayart} initiated an  $\mathcal{H}_p$-theory of ordinary Dirichlet series. Recall that  $H_p(\mathbb{T}^\infty)\,, \, 1 \leq p \leq \infty$, is the  closed subspace  of all $f \in L_p(\mathbb{T}^\infty)$
 which have a Fourier transforms $\hat{f}: \mathbb{Z}^{(\mathbb{N})} \rightarrow \mathbb{C}$ supported on $\mathbb{N}_0^{(\mathbb{N})}$. Then the  Banach spaces $\mathcal{H}_p, 1 \leq p \leq \infty$, consist of all  ordinary Dirichlet series $\sum a_n n^{-s}$ for which there is some (unique) $f \in H_p(\mathbb{T}^\infty)$ such that $a_n = \hat{f}(\alpha)$ for all $\alpha \in \mathbb{N}_0^{(\mathbb{N})}$ with  $\mathfrak{p}^\alpha = n$.
 Together with the norm $\|D\|_p = \|f\|_p$ this leads to Banach spaces. Hence by the very definition the so-called
 Bohr transform
 \begin{equation}\label{Bohrtrafo}
 \mathfrak{B}: H_p(\mathbb{T}^\infty) \to \mathcal{H}_p, \,\,\,f \mapsto D\,,
 \end{equation}
 where $D = \sum a_n n^{-s}$ with $a_n = \hat{f}(\alpha)$ for all $\alpha \in \mathbb{N}_0^{(\mathbb{N})}$ with  $\mathfrak{p}^\alpha = n$, is an isometric linear bijection.
  For every Dirichlet polynomial $D = \sum_{n \leq x} a_n n^{-s}$ we for $1 \leq p < \infty$ have that
 \begin{equation} \label{norm}
 \|D\|_p = \lim_{T \to \infty} \frac{1}{2T}  \Big(\int_{-T}^T
 \big|
  \sum_{n \leq x} a_n n^{-it} \big|^p dt\Big)^{\frac{1}{p}}\,,
 \end{equation}
 and hence in this case it turns out that  the completion of  the linear space of all  Dirichlet polynomials under this norm gives
 precisely  the Banach space $\mathcal{H}_p$.

 Obviously, a Dirichlet series $D = \sum a_n n^{-s}$ belongs to the Hilbert space $\mathcal{H}_2$ if and only if
 the sequence  $(a_n)$ of its Dirichlet coefficients
 belongs to $\ell_2$.

 Moreover, as a consequence of Theorem \ref{Bohr} it can be shown that $\mathcal{H}_\infty$
 and $\mathcal{D}_\infty$ coincide as Banach spaces,
 \begin{equation} \label{bruce}
 \mathcal{H}_\infty =\mathcal{D}_\infty \,\,\, \text{isometrically}\,.
 \end{equation}
  This  fundamental fact  was  observed in \cite{HLS} (see also \cite[Corollary 5.3]{Defant}), and we will come back to it in Section \ref{H_B}.

 In contrast to the situation for $\mathcal{H}_\infty$, arbitrary  Dirichlet series in $\mathcal{H}_p, 1 \leq p < \infty$, only converge pointwise  on $[\text{Re} > 1/2]$ (even  absolutely), and in general this half-plane
 can not be replaced by a bigger one
 (see e.g \cite[Remark 12.13 and Theorem 12.11]{Defant}). We remark that by a result from \cite{HedenmalmSaksman}
 each Dirichlet series $D$ in $\mathcal{H}_2$ converges almost everywhere on the abscissa $[\text{Re} = 1/2]$
 (and consequently also every $D \in \mathcal{H}_p,\, 2 \leq p < \infty$), whereas
 Bayart
 in \cite{Bayart2} gave an example of a Dirichlet series in $\mathcal{H}_\infty$ that diverges at every point of the imaginary axis.
 It seems that
   for  Dirichlet series in $\mathcal{H}_p$, $1 \leq p < 2$,
  there is no such precise knowledge on pointwise convergence   on the abscissa $[\text{Re} = 1/2]$.

 The horizontal
  translation  of a Dirichlet series $D = \sum a_n n^{-s}$ about $u>0$ is defined to be the Dirichlet series
$$D_{u}:=\sum \frac{a_{n}}{n^u} n^{-s}.$$
Given $D \in \mathcal{H}_p$,
 then $\Bcal(f*p_{u})=D_{u}$, which in particular shows that  $D_{u}\in \mathcal{H}_{p}$.

 But more can be said: For each $1 \leq p,q < \infty$
and $u>0$ there is a constant $E=E(u,p,q)$ such that for each $D \in \mathcal{H}_p$ we have
\begin{equation} \label{hypo}
D_{u} \in \mathcal{H}_q \,\,\, \text{ and }\,\,\,\|D_{u}\|_q \leq E  \|D\|_p\,.
\end{equation}
This result is basically due to Bayart \cite{Bayart}, for a self-contained proof see
\cite[Theorem 12.9]{Defant}; we refer to this  fact as the the 'hypercontractivity' of the Hardy spaces $\mathcal{H}_p$.

 \subsection{Vertical limits}
 Given a   Dirichlet series $D = \sum a_n n^{-s}$, then we call
the Dirichlet series $D_{z}(s):=\sum \frac{a_{n}}{n^z} n^{-s}$ the   translation of $D$ about $z \in \mathbb{C}$,
and  each Dirichlet series of the form $$D^{\chi}=\sum a_{n} \chi(n) n^{-s}\,,\,\, \chi \in \Xi$$
is said to be a vertical limits of $D$. Examples are vertical translations
$D_{i\tau}=\sum a_{n} n^{-i\tau}  n^{-s}$ with $\tau \in \mathbb{R}$,
and the terminology is explained by the fact that each vertical limit  may be approximated by  vertical translates. More precisely, given $D = \sum a_n n^{-s}$ which converges absolutely on the right half-plane , for every  $\chi \in \Xi$ there is a sequence $(\tau_{k})_{k} \subset \R$ such that $D_{i\tau_{k}}$ converges to $D^{\chi}$ uniformly on $[Re>\varepsilon]$ for all $\varepsilon>0$.
 Assume conversely that for  $(\tau_{k})_{k} \subset \R$ the  vertical translations $D^{i\tau_k}$ converge
 uniformly on $[Re>\varepsilon]$ for every  $\varepsilon>0$
 to a holomorphic function $f$ on $[Re>0]$. Then there is $\chi \in \Xi$ such that
    $f(s)= \sum_{n=1}^\infty a_n \chi(n) n^{-s}$
    for all $s \in [Re>0]$\,. For this see \cite[Section 4.1]{DefantSchoolmann}.

    It is simple to show that each vertical limit $D^\chi$ belongs to $\mathcal{H}_p$ if and only if
    $D$ does, and the norms remain the same (apply Bohr transform and use the rotation invariance of the Lebesgue measure on $\mathbb{T}^\infty$).

 Finally, we recall that every function $f\in L_{1}(\mathbb{T}^\infty)$ for almost all $z \in \mathbb{T}^\infty$ allows a locally Lebesgue integrable 'restriction' $f_{z} \colon \R \to \C$ such that $f_{z}(t)=f(z \beta(t))$ for almost all $t\in \R$ (see \cite[Lemma 3.10]{DefantSchoolmann}). More explicitly, for almost all $z \in \mathbb{T}^\infty$ the function
 \begin{equation} \label{flowflow}
 f_z : \mathbb{R}  \to \mathbb{C},\,\, f_z(t) = f\big( \mathfrak{p}^{-it} z \big)
 \end{equation}
 is locally integrable.

 Given $f \in H_1(\mathbb{T}^\infty)$, the  family  $(f_z)_{z \in \mathbb{T}^\infty}$ of functions on $\mathbb{R}$ form a sort of  bridge to tools from Fourier analysis on $\R$. The following simple lemma
   (see \cite{DefantSchoolmann3}) shows how pointwise convergence   on $\mathbb{T}^\infty$ is related  with pointwise convergence  on $\R$.

\begin{Lemm} \label{hedesaks*} Let $f_n, f \in H_1(\mathbb{T}^\infty)$. Then the following are equivalent:
\begin{itemize}
\item[(1)]
$\lim_{n\to \infty} f_n(z) = f(z)$ \,\,\, \text{for almost all $z \in \mathbb{T}^\infty$}
\vspace{1mm}
\item[(2)]
$\lim_{n\to \infty} (f_n)_z(t)= f_z(t)$ \,\,\,
\text{for almost all $z \in \mathbb{T}^\infty$ and for  almost all $t\in \R$\,.}
\end{itemize}
In particular, if all $f_n$ are polynomials and $D_n \in \mathcal{H}_1$
 are the Dirichlet series  associated  to $f_n$ under Bohr's transform, then $(1)$  and $(2)$ are equivalent to each of the
 following two further statements:
\begin{itemize}
\item[(3)]
$\lim_{n\to \infty} D^\chi_n(0)= f(\chi (\mathfrak{p}))$ \,\,\,
\text{for almost all $\chi \in \Xi$}
\vspace{1mm}
\item[(4)]
$\lim_{n\to \infty} D^\chi_n(it)= f\Big( \frac{\chi(\mathfrak{p})}{\mathfrak{p}^{it}}\Big)$ \,\,\,
\text{for almost all $\chi \in \Xi$ and for  almost all $t\in \R$}\,.
\end{itemize}
\end{Lemm}

     \section{Helson meets Menchoff-Rademacher}
     Let us come back to Helson's theorem from Theorem \ref{Helson} which states that almost all vertical
     limits $D^\chi$ for Dirichlet series $D \in \mathcal{H}_2$ converge on the right half-plane. Helson's original proof is mainly based on classical tools from Fourier analysis, as e.g. the
     Fourier inversion theorem or  the Hausdorff-Young inequality.

     How does the relevant maximal inequality for Helson's theorem look like?
    Inspired by an idea of Bayart \cite{Bayart} we  isolate  this maximal inequality using Kantorovitch's maximal inequality  \eqref{Kantorovitch} from the Menchoff-Rademacher theorem.

     \begin{Theo} \label{HelsonII}
     For every $u > 0$  there is a constant $C = C(u) >0$ such that for every
     $D \in \mathcal{H}_2$ we have
     \begin{equation*}
 \Big(\int_{\Xi} \sup_{x} \Big| \sum_{n=1}^x \frac{a_n}{n^u}\chi(n)\Big|^2\Big)^{\frac{1}{2}} d\chi  \le C
 \|D\|_{2}.
 \end{equation*}
 Equivalently,  for every
     $f \in H_2(\mathbb{T}^\infty)$
      \begin{equation*}
      \Big( \int_{\mathbb{T}^\infty} \sup_{x} \Big| \sum_{\mathfrak{p}^\alpha < x}
      \frac{\widehat{f}(\alpha)}{\mathfrak{p}^{u\alpha } }
      z^\alpha \Big|^{2} dz \Big)^{\frac{1}{2}} \le C\|f\|_{2}.
\end{equation*}
\end{Theo}

          \begin{proof}
          Clearly, the functions $\Xi \to \mathbb{C}, \chi \mapsto \chi(n)$ form an orthonormal
          system in $L_2(\Xi)$. Hence by Theorem \ref{MenchoffRademacher}
          for all $a_1, \ldots, a_x \in \mathbb{C}$
          \begin{align*}
          \Big( \int_{\Xi} \sup_{x} \Big| \sum_{n=1}^x \frac{a_n}{n^u}\chi(n)\Big|^{2} d\chi \Big)^{\frac{1}{2}}
          \ll \Big\|\big(a_n \frac{\log n}{n^{u}}\big)\Big\|_2
                     \ll \big\|(a_n)\big\|_2  =  \|D\|_{2}\,.
          \end{align*}
         The second assertion is an obvious reformulation through Bohr's transform \eqref{Bohrtrafo}.
                  \end{proof}

      But applying the hypercontractivity estimate from \eqref{hypo}
     we even get all this for the much  larger class of Dirichlet series in  $\mathcal{H}_1$
      (apart from the maximal inequality this was  observed  in \cite[Theorem 6]{Bayart}).

     \begin{Theo} \label{MRK}
     For every $u > 0$  there is a constant $C = C(u) >0$ such that for every
     $D \in \mathcal{H}_1$ we have
     \begin{equation*}
\int_{\Xi} \sup_{x} \Big| \sum_{n=1}^x \frac{a_n}{n^u}\chi(n)\Big| d\chi  \le C
 \|D\|_{1}.
 \end{equation*}
 Equivalently,  for every
     $f \in H_1(\mathbb{T}^\infty)$
      \begin{equation*}
     \int_{\mathbb{T}^\infty} \sup_{x} \Big| \sum_{\mathfrak{p}^\alpha < x}
      \frac{\widehat{f}(\alpha)}{\mathfrak{p}^{u\alpha } }
      z^\alpha \Big| dz  \le C\|f\|_{1}.
\end{equation*}
\end{Theo}

      \begin{proof}
      Take $u >0$ and $D \in \mathcal{H}_1$.  Then we know from \eqref{hypo} that $D_{u} \in \mathcal{H}_2$
      and $\|D_{u}\|_2 \leq E(u) \|D_{u}\|_1$. Hence by
           Theorem \ref{HelsonII}
                  \begin{align*}
            \int_{\Xi} \sup_{x} \Big| \sum_{n=1}^x \frac{a_n}{n^u}\chi(n)\Big| d\chi
            &
           \leq
          \Big( \int_{\Xi} \sup_{x} \Big| \sum_{n=1}^x \frac{a_n}{n^u}\chi(n)\Big|^{2} d\chi \Big)^{\frac{1}{2}}
          \\&
           \le  C(u) \|D_{u}\|_{2} \leq C(u)E(u) \|D\|_{1}\,,
          \end{align*}
                    which is the inequality we aimed for. The second assertion again follows from Bohr's transform
          \eqref{Bohrtrafo}.
          \end{proof}

          To derive from the preceding  maximal inequalities  results on pointwise convergence is standard, and an argument is formalized in \cite[Lemma 3.6]{DefantSchoolmann3}.
Then  we together with Lemma \ref{hedesaks*} conclude the following improvement  of Theorem \ref{Helson} (taken from \cite[Corollary 2.3]{DefantSchoolmann3}).

    \begin{Theo} \label{Helsonaxis}
    Let  $f \in H_1(\mathbb{T}^\infty)$ and
    $D= \sum a_{n}n^{-s}\in \mathcal{H}_1$ its associated Dirichlet series under the Bohr transform. Then
    for all $u>0$
    \[
    f(z)=  \lim_{x \to \infty }\sum_{ \mathfrak{p}^\alpha < x }\frac{\hat{f}(\alpha)}{\mathfrak{p}^{u \alpha}} z^\alpha\
    \]
    almost everywhere  on $\mathbb{T}^\infty$. Equivalently, almost all vertical limits
    $D^\chi$ converges  on the right half-plane. More precisely, there is a null set $N\subset \Xi$ such that, if $\chi \notin N$, then we for all $u+it \in [\text{Re}>0]$  have
    $$D^{\chi}(u+it)=f_{\chi(\mathfrak{p})}*P_{u}(t).$$
     \end{Theo}

There is another interesting aspect of this theorem. A result of Bayart from
        \cite{Bayart}  (see e.g.  \cite[Theorem 12.11]{Defant} for an explicit formulation)
        states that the best $u>0$ such that
        $\sum_{ \alpha }
    \frac{|\widehat{f}(\alpha)|}{\mathfrak{p}^{ u \alpha}} < \infty $
                for all $f \in H_1(\mathbb{T}^\infty)$, equals $1/2$.
                In particular, for all $u \le 1/2$ there is
                $f \in H_1(\mathbb{T}^\infty)$ such that the equality '$f(z) = \sum_{\alpha}\frac{\widehat{f}(\alpha)}{\mathfrak{p}^{ u \alpha}}z^\alpha$ a.e.' fails.
                Note that in contrast to this, the limit in Theorem \ref{Helsonaxis} exists for all $u>0$.

    \section{Helson meets Riemann} \label{Riemann}
    Let us come back to Theorem \ref{riemann} which shows that Riemann's conjecture holds for allmost all vertical limits of the zeta series. The following result from \cite[Theorem 4.6]{HLS} (here even proved for Dirichlet series
    in $\mathcal{H}_1$) is an improvement.

    \begin{Theo} \label{RiemannH1}
 Let $D \in \mathcal{H}_1$ have  completely
  multiplicative coefficients $a_n$. Then almost all vertical limits $D^\chi$  converge on $[\text{Re} >0]$ to a zero free function.
  \end{Theo}

  \begin{proof}
  Assume first $D\in \mathcal{H}_2$, and recall the definition of the M\"obius function $\mu:\mathbb{N} \to \{1,-1,0\}$
  \[
  \mu(n)
  =
  \begin{cases} 1& \text{ if }  n=1 \\
  (-1)^k        & \text{ if }  n= \mathfrak{p}_1\ldots \mathfrak{p}_k\\
  0     & \text{ else }
    \end{cases}
    \]
    Then for every character $\chi$ a formal(!) caculation shows that
    \[
    \Big( \sum a_n\chi(n) n^{-s}\Big) \ast \Big( \sum a_n  \chi(n)\mu (n)n^{-s}\Big) =1\,,
    \]
    where $\ast$ stands for the Cauchy product. Almost all vertical linits of these two Dirichlet series in $\mathcal{H}_2$ converge on the right half-plane , hence they there are zero free.
    In a second step, let  $D \in \mathcal{H}_1$. Then, given $n \in \mathbb{N}$,  by hypercontractivity \eqref{hypo}
    the  horizontal translation $D^{1/n} \in \mathcal{H}_2$,
    and hence there is a null set $\Xi_n$ such that all $(D^{1/n})^\chi, \chi \in C\Xi_n $ are zero free  on the right half-plane .
    But then for $\chi \in C \big( \bigcup \Xi_n\big)$ all Dirichlet series $(D^{1/n})^\chi, n \in \mathbb{N}$
    are zero free  on the right half-plane  which implies that all $D^\chi$ are.
      \end{proof}

      If we apply this result to the translated zeta series from \eqref{zeta}, then the proof of Theorem \ref{riemann} is immediate.

    \section{Helson meets Carleson-Hunt} \label{HelCar}
    It is well-known that Theorem \ref{CarlesonHunt}  extends to the case $1 < p < \infty$, a result then usually called Carleson-Hunt theorem.

    \begin{Theo} \label{CarlesonHunt_p}
    Let $1 < p < \infty$.
    The Fourier series
   $
   f(z) = \sum_{n=0}^\infty \hat{f}(n) z^n
   $
   of every   $f \in H_p(\mathbb{T})$ converges almost everywhere on $\mathbb{T}$. Moreover, there is a (best)
   constant $\text{CH}_p >0$ such that
   \begin{equation*}
\Big( \int_{\mathbb{T}} \sup_{x} \Big| \sum_{n < x} \widehat{f}(n) z^n \Big|^{p} dz \Big)^{\frac{1}{p}} \le \text{CH}_p \|f\|_{p}\,.
\end{equation*}
\end{Theo}
    And all this fails for $p=1$.
    A reasonable question would be to ask for almost everywhere  pointwise convergence on $\mathbb{T}^\infty$ of the Fourier
    series $\sum_\alpha \hat{f}(\alpha)z^\alpha$ whenever $f \in H_p(\mathbb{T}^\infty)$. But how do we order the multi indices? If we consider absolute convergence of $\sum_\alpha \hat{f}(\alpha)z^\alpha$, then we necessarily need to assume  $(\hat{f}(\alpha)) \in \ell_1$.

     For $p=2$ the following result  was proved by
    Hedenmalm-Saksman in \cite[Theorem 1.5]{HedenmalmSaksman}, and later it was  extended
     by Duy in \cite{Duy}  to the scale of $1 < p <\infty$ and arbitrary general Dirichlet series
     (see also  \cite{DefantSchoolmann2}) .

  \begin{Theo} \label{one}
  Let  $1<p<\infty$. Then
    for every $f \in H_{p}(\mathbb{T}^\infty)$ we have
\begin{equation*}
\Big( \int_{\mathbb{T}^\infty} \sup_{x} \Big| \sum_{\mathfrak{p}^\alpha < x} \widehat{f}(\alpha) z^\alpha \Big|^{p} dz \Big)^{\frac{1}{p}} \le \text{CH}_p \|f\|_{p}.
\end{equation*}
  Equivalently,
  for every $D = \sum a_n n^{-s} \in \mathcal{H}_p$
\begin{equation*}
\Big( \int_{\Xi} \sup_{x} \Big| \sum_{n=1}^x a_n \chi(n)\Big|^{p} d\chi \Big)^{\frac{1}{p}} \le \text{CH}_p \|D\|_{p}.
\end{equation*}
  \end{Theo}
  Clearly, here the particular case $p=2$  proves  Theorem \ref{HelsonII} for $u=0$ which means  a strong improvement.
In fact the proof from \cite{Duy} (and also \cite{DefantSchoolmann2}) follows the strategy invented by  Hedenmalm-Saksman
in \cite{HedenmalmSaksman}; this strategy applies the orginal Carleson-Hunt theorem for functions in one variable through a magic trick invented by Fefferman in \cite{Fefferman}. Observe, that if the function $f$ in the preceding theorem  only depends on the first variable, then we recover the full one variable case from Theorem \ref{CarlesonHunt_p}.  But note that
Theorem \ref{one} does not cover Theorem \ref{MRK}, since for $p=1$ it definitely fails.

Again we may deduce convergence theorems. Combining Theorem \ref{one} and adding Lemma \ref{hedesaks*},   we conclude the following
result from \cite{DefantSchoolmann2} (see again \cite[Theorem 1.4]{HedenmalmSaksman} for the case $p=2$).

    \begin{Theo} \label{Helsonaxis2}
    Let  $f \in H_p(\mathbb{T}^\infty), \, 1 < p < \infty$ and
    $D= \sum a_{n}n^{-s}\in \mathcal{H}_p$ its associated Dirichlet series under the Bohr transform. Then
    \[
    f(z)=  \lim_{x \to \infty }\sum_{ \mathfrak{p}^\alpha < x }\hat{f}(\alpha) z^\alpha\
    \]
    almost everywhere  on $\mathbb{T}^\infty$. In particular, for almost all $\chi \in \Xi$
    \begin{itemize}
    \item[(1)]
    $D^\chi$ converges in $s=0$ to $f(\chi(p))$,
     \item[(2)]
     $D^\chi$ converges almost everywhere on the imaginary axis, and for almost all $t \in \R$ we have $D^\chi(it) = f_{\chi(p)}(t)$,
     \item[(3)]
      $D^\chi$ converges everywhere on the right half-plane, and we for all $u>0$ almost everywhere on $\R$ have $D^{\chi}(u+it)=f_{\chi(p)}*P_{u}(t)$.
      \end{itemize}
         \end{Theo}

     This is a considerably  strong  extension of  Theorem \ref{Helson}. Since the Carleson-Hunt Theorem \ref{CarlesonHunt_p}
     on pointwise convergence does not hold for $p=1$, for this case the first two  conclusions  are false
     (to see this for the  second one use Lemma \ref{hedesaks*}) whereas the last one  holds by Theorem \ref{Helsonaxis}.

\section{Helson meets Bohr} \label{H_B}
Are Helson's Theorem \ref{Helson} and Bohr's Theorem \ref{Bohr} linked? Can Helson-like arguments be used
to prove the important Banach space identity $\mathcal{H}_\infty = \mathcal{D}_\infty$ from \eqref{bruce}?
Both questions have a positive answer,
and  this will be a consequence of the following  maximal inequality from \cite{DefantSchoolmann2} which
is a variant of Theorem \ref{one}.

  \begin{Theo} \label{twotwo} For every $u >0$ there is a constant $C>0$ such that for for every $1 \leq p <\infty$ there
  and  every $f \in H_{p}(\mathbb{T}^\infty)$ we have
\begin{equation*}
\Big( \int_{\mathbb{T}^\infty} \sup_{x} \Big| \sum_{\mathfrak{p}^\alpha \leq x}
\frac{\widehat{f}(\alpha)}{\mathfrak{p}^{ u \alpha}} z^\alpha \Big|^{p} dz \Big)^{\frac{1}{p}} \le C\|f\|_{p}.
\end{equation*}
    Equivalently,
  for
   every $D = \sum a_n n^{-s} \in \mathcal{H}_p$
\begin{equation}  \label{twoA}
\Big( \int_{\Xi} \sup_{x} \Big| \sum_{n=1}^x \frac{a_n}{n^u}\chi(n)\Big|^{p} d\chi \Big)^{\frac{1}{p}} \le C\|D\|_{p}.
\end{equation}
\end{Theo}

What  are  the differences of Theorem \ref{twotwo} and   Theorem \ref{one}\,? There are two. Admitting
translations along some $u>0$,  Theorem \ref{twotwo}
 holds for $p=1$, whereas Theorem \ref{one} does not. So in this case the preceding result recovers
 Theorem \ref{MRK}.
  But secondly, and more important, the constant in the maximal inequality from Theorem \ref{twotwo} does not depend on $p$.
We remark that another application of  Lemma \ref{hedesaks*}  again allows
us, now with a different argument, to  recover Theorem \ref{Helsonaxis}.

Let us indicate what is happening  whenever  in Theorem \ref{twotwo} the parameter  $p$ tends to $\infty$.
Obviously, we get that there is a constant $C >0$ such that for each $u> 0$ and $D = \sum a_n n^{-s} \in \mathcal{H}_\infty$
\begin{equation} \label{remain}
\Big\| \sup_{x} \Big| \sum_{n=1}^x \frac{a_n}{n^u}\chi(n) \Big| \Big\|_{L_\infty(\Xi)} \leq C \|D\|_{\mathcal{H}_{\infty}}\,.
\end{equation}
The following corollary of this inequality recovers part of
Theorem \ref{Bohr} and the identity  from  \eqref{bruce}. It shows up to which
amount the preceding Helson type Theorem \ref{twotwo} still reflects Bohr's original ideas.

\begin{Coro} \label{kreis}
 For each  $D= \sum a_{n}n^{-s} \in \mathcal{H}_\infty$ we have
  \[
  D \in \mathcal{D}_\infty \,\,\, \text{ and  } \|D\|_{\mathcal{D}_\infty} \leq  \|D\|_{\mathcal{H}_\infty}\,,
  \]
  and moreover $D$ for each $u >0$ converges uniformly on
   $[\text{Re}> u]$.
   \end{Coro}

\begin{proof}
We fix some $D = \mathfrak{B}(f) \in \mathcal{H}_\infty$.
From  the density of the  Kronecker flow from \eqref{flow} (i.e. Kronecker's approximation theorem)
and   \eqref{remain} we deduce that for each $y>0$ and each $\psi \in \Xi$
\begin{align*}
  \sup_{t \in \mathbb{R}} \Big|\sum_{n <y}  \frac{a_n}{n^u}\psi(n)n^{-it}  \Big|
  &
       =  \sup_{z \in \mathbb{\mathbb{T}^\infty}} \Big|\sum_{\mathfrak{p}^\alpha <y}  \frac{a_{\mathfrak{p}^\alpha}}
  {\mathfrak{p}^{u\alpha}} \psi(\mathfrak{p}^\alpha)
   z^\alpha  \Big|= \Big\| \sum_{n=1}^y \frac{a_n}{n^u}\psi(n)\chi(\cdot)  \Big\|_{L_\infty(\Xi)}
    \\&
      = \Big\| \sum_{n=1}^y \frac{a_n}{n^u}\chi(\cdot)  \Big\|_{L_\infty(\Xi)}
    \leq C \|D\|_{\mathcal{H}_{\infty}}\,.
    \end{align*}
    Hence by the so-called  Bohr-Cahen formula on uniform convergence
  (see e.g. \cite[Proposition 1.6]{Defant}) we see that
  $D^\psi$ for each $u >0$ converges uniformly on
   $[\text{Re}> u]$, and consequently
  $D^\psi \in \mathcal{D}_\infty$ for all $\psi \in \Xi$. In  particular this holds for $D$ itself. By \eqref{flowflow} we know that the functions
  $f: \R \to \mathbb{C}, f_{\chi}(t) = f(\chi(\mathfrak{p}) \mathfrak{p}^{-it})$ are integrable for almost all $\chi \in \Xi$. Comparing Fourier coefficients we conclude  that
  \[
  D^\chi(u +it) = (f_{\chi(\mathfrak{p})} \ast P_u)(t)
  \]
  for almost all $\chi\in \Xi$ and all $u+it \in [\text{Re}>0]$. Observe that for all $\chi \in \Xi$
   \[
  \|D\|_{\mathcal{D}_\infty} = \|D^{\chi}\|_{\mathcal{D}_\infty}\,.
  \]
 Indeed, denote the $N$th partial sum of $D$ by $D_N$, and fix some  $u >0$. Since we already checked that
  $D$ and $D^\chi$ converge uniformly on
   $[\text{Re}> u]$, another application of \eqref{flow} yields
 \[
 \|D_u\|_{\mathcal{D}_\infty}
 =\lim_{N \to \infty} \| (D_N)_u\|_{\mathcal{D}_\infty}
 =\lim_{N \to \infty} \| (D^\chi_N)_u\|_{\mathcal{D}_\infty} = \|D_u\|_{\mathcal{D}_\infty}\,.
 \]
 Alltogether we get that there is some $\chi$ such that
 \[
 \|D\|_{\mathcal{D}_\infty} = \|D^\chi\|_{\mathcal{D}_\infty}
 = \|f_{\chi(\mathfrak{p})} \ast P_u\|_{L_\infty(\R)}
 \leq \|f_{\chi(\mathfrak{p})}\|_{L_\infty(\R)} = \|f\|_{H_\infty(\mathbb{T}^\infty)} =\|D\|_{\mathcal{H}_{\infty}}\,,
  \]
  the conclusion.
 \end{proof}

  We emphasize  that this 'outcome' of Theorem \ref{twotwo} is - in two respects - weaker than what is known from Theorem \ref{Bohr} and \eqref{bruce}.
  Indeed, \eqref{bruce} is an isometric identity and not just a contractive inclusion.

  We refer to  \cite[Corollary 5.3]{Defant} and \cite{HLS}, where the Banach space identity  \eqref{bruce} is proved using Bohr's Theorem \ref{Bohr} and going  an
  (independently interesting) 'detour' through
bounded holomorphic functions on the open unit ball of $c_0$. In the following we indicate an argument
for \eqref{bruce}
which avoids this detour, and uses instead  Corollary \ref{kreis} (so a
 'Helson-like argument') and
Theorem \ref{Bohr} (first statement). This idea goes back to \cite[Proof of Theorem 4.1]{DefantSchoolmann}.
\begin{proof}[A proof  of \eqref{bruce} using no infinite dimensional holomorphy:]
From Corollary \ref{kreis} we know that $\mathcal{H}_\infty \subset \mathcal{D}_\infty$,
and $\|D\|_{\mathcal{H}_\infty} \leq  \|D\|_{\mathcal{H}_\infty}$.

   Conversely, take $D \in \mathcal{D}_\infty$ with its associated sequence of $N$th partial sums $D^N$,
and look again for each $u>0$ at the sequence $(f^{N}_{u})_N$ in $H_\infty(\mathbb{T}^\infty)$ which corresponds to the sequence $(D^{N}_{u})$
under the Bohr transform. By \eqref{flow} and Theorem \ref{Bohr} (first statement) this is a Cauchy sequence in the Banach space $\mathcal{H}_\infty$
which has a limit $f_u$ satisfying $\|f_u\|_{\mathcal{H}_\infty} \leq \|D\|_{\mathcal{D}_\infty}$.
Now recall that the unit ball of $L_\infty(\mathbb{T}^\infty)$ endowed with its weak star topology is metrizable and compact. Hence $(f_{1/n})_n$ has a weak star convergent subsequence with limit $f \in L_\infty(\mathbb{T}^\infty)$ and
$\|f\|_{\mathcal{H}_\infty}  \leq \|D\|_{\mathcal{D}_\infty}$. Then a simply argument shows that
$f \in H_\infty(\mathbb{T}^\infty)$ and $\mathfrak{B} (f) = D$, which finishes the argument.
\end{proof}

     \section{Helson  meets Ces\`{a}ro} \label{fail}
\noindent           For $f \in H_1(\mathbb{T})$ is well-known that
\begin{equation} \label{cesaro}
f(z) = \lim_{N\to \infty} \frac{1}{N} \sum_{k=0}^{N-1} \sum_{n\leq k} \hat{f}(n) z^n
= \lim_{N \to \infty} \sum_{n=0}^{N-1} \widehat{f}(n) \left(1-\frac{n}{x}\right) z^{n}
\end{equation}
holds in the $H_1$-norm as well as pointwise almost everywhere on $\mathbb{T}$. In other words, Ces\`{a}ro
means (arithmetic means) are perfectly adapted to the summation of Fourier series of
integrable functions in one variable.

What about infinitely many variables? We know from  Theorem \ref{Helsonaxis}  that for any $D \in \mathcal{H}_1$ the sequence of Ces\`{a}ro means
$$\bigg( \frac{1}{N} \sum_{k=0}^{N-1} \sum_{n\le k} a_n \chi(n) n^{-s} \bigg)_{N\in \mathbb{N}}$$
for almost all $\chi \in \Xi$ converge on $[Re >0]$, i.e. the Ces\`{a}ro means of almost all
vertical limits $D^\chi$ converge on $[Re >0]$. And in  Theorem \ref{Helsonaxis2}  we even saw that
for $D \in \mathcal{H}_p, 1 < p < \infty$, allmost all vertical limits converge
 on the imaginary axis. But since this is false in the case $p=1$, the following question seems natural:\\

    \noindent {\bf Question:}  \label{question1}
    {\it
    Is it true that, given  $D \in \mathcal{H}_1$, the limits
    \[
    \lim_{N \to \infty}  \frac{1}{N} \sum_{k=0}^{N-1} \sum_{n \le k} a_n \chi(n) n^{-it}
    \]
exist for almost all $\chi \in \Xi$ and $t \in \mathbb{R}$? Or, by Lemma \ref{hedesaks*}
 equivalently, do we for $f \in H_1(\mathbb{T}^\infty)$ have that
\begin{equation} \label{labalaba}
    f(z) = \lim_{N \to \infty}  \frac{1}{N}  \sum_{k=0}^{N-1} \sum_{\mathfrak{p^\alpha} \le k} \hat{f}(\alpha) z^\alpha
    \end{equation}
almost everywhere on $\mathbb{T}^\infty$?\\
}
\\
Let us show, with an idea going back to Hardy  and Riesz from \cite[Thorem 21, p. 36]{HardyRiesz},  that the answer is negative. Indeed, for $f \in H_1(\mathbb{T}^\infty)$  and every $N$
\begin{align*} \label{compare}
\frac{1}{N}  \sum_{k=0}^{N-1} \sum_{\mathfrak{p^\alpha} \leq k} \hat{f}(\alpha) z^\alpha
= \sum_{\mathfrak{p^\alpha} < N} \hat{f}(\alpha) \big(1-\frac{\mathfrak{p}^\alpha}{N}\big)  z^\alpha \,.
\end{align*}
Hence, if the answer to the above question is affirmative, then for every $f \in H_1(\mathbb{T})$ we get that
\begin{equation} \label{nekar}
    \lim_{N \to \infty} \sum_{2^j < N} \hat{f}(j) \big(1-\frac{2^j}{N}\big)  z^j
     = f(z)
     \,\,\,\, \text{almost everywhere on $\mathbb{T}$\,,}
    \end{equation}
and a straight forward calculation shows
$$\sum_{2^{j}<2^{N+1}} \widehat{f}(j) z^{j} (2^{N+1}-2^{j})-\sum_{2^{j}<2^{N}} \widehat{f}(j) z^{j} (2^{N}-2^{j})=2^{N}\left(\widehat{f}(N)z^{N}+\sum_{j=0}^{N-1} \widehat{f}(j)z^{j}\right).$$
Consequently,
\begin{equation*}
\sum_{j=0}^{N-1} \widehat{f}(j)z^{j}=2\sum_{2^{j}<2^{N+1}} \widehat{f}(j) z^{j} \big(1-\frac{2^{j}}{2^{N+1}}\big)-\sum_{2^{j}<2^{N}} \widehat{f}(j) z^{j} \big(1-\frac{2^{j}}{2^{N}}\big)-\widehat{f}(N)z^{N},
\end{equation*}
and so using (\ref{nekar}) we conclude that
 for almost all $z \in \mathbb{T}$
 \[
\sum_{j=0}^{\infty} \hat{f}(j) z^j  = f(z)  \,,
 \]
a contradiction.

    \section{Helson meets Hardy-Riesz}

A  proper substitute for Ces\`{a}ro summation  (even within the setting of general Dirichlet series) was already suggested by Hardy and Riesz in \cite{HardyRiesz} where the first author writes: {\it ...it appeared from the investigations of Riesz that these arithmetic means are not so well adapted to the study of the series
as certain other means in a somewhat different manner. These 'logarithmic means', ......, have generalisations especially adapted to the study of general series $\sum a_n e^{-\lambda_n s}$.}
\\

Let $D=\sum a_n n^{-s} \in \mathcal{H}_1$  and $f\in H_1(\mathbb{T}^\infty)$
be associated under the Bohr transform $\mathfrak{B}$.
 Then the first Riesz mean (or logarithmic mean) of $D$ of  length $x>0$ is given by the Dirichlet polynomial
$$R_{x}(D):=\sum_{\log n<x} a_{n} \left(1-\frac
{\log n}{x}\right)n^{-s}\,,$$
and analogously  the analytic polynomial
$$R_{x}(f)= \mathfrak{B}^{-1}\big( R_{x}(D)\big)
=\sum_{\log \mathfrak{p}^\alpha <x} \widehat{f}(\alpha) \left(1-\frac{\log \mathfrak{p}^\alpha}{x}\right) z^\alpha$$
is  the first Riesz mean (logarithmic mean) of length $x>0$ of
$f$.\\

The following maximal inequality  is the main result from \cite[Theorem 2.1]{DefantSchoolmann3}
(even for  Riesz means of general Dirichlet series of arbitrary order $k>0$),
and it rules the logarithmic summation of functions on the infinite dimensional torus and ordinary Dirichlet series. Compare this with Theorem \ref{one} which handles  usual summation -- but only for  $p>1$.\\

Recall that the weak $L_{1}(\mathbb{T}^\infty)$-space $L_{1,\infty}(\mathbb{T}^\infty)$ consists of all measurable functions $f\colon \mathbb{T}^\infty \to \C$ for which  there is a constant $C>0$ such that for all $\alpha>0$
\begin{equation*} \label{Lorentz}
m\left(\left\{ z \in \mathbb{T}^\infty \mid |f(z)|>\alpha \right\} \right)\le \frac{C}{\alpha}.
\end{equation*}
With $\|f\|_{1,\infty}:= \inf C$, the space $L_{1,\infty}(\mathbb{T}^\infty)$ becomes a quasi Banach space.

\begin{Theo}\label{maximalineqtypeI} The sub-linear operator
$$T(f)(w):=\sup_{x>0} |R_{x}(f)(w)|=\sup_{x>0} \bigg|\sum_{\log \mathfrak{p}^{\alpha} <x} \widehat{f}(\alpha) \left(1-\frac{\log \mathfrak{p}^{\alpha}}{x}\right)z^{\alpha}\bigg|$$
is bounded from $H_{1}(\mathbb{T}^\infty)$ to $L_{1,\infty}(\mathbb{T}^\infty)$, and from $H_{p}(\mathbb{T}^\infty)$ to $L_{p}(\mathbb{T}^\infty)$, whenever $1<p\le \infty$.
\end{Theo}

In the following section we briefly sketch the proof of this result which seems to have some facets of independent interest.\\

But again we first want to formulate a consequence on pointwise Ces\`{a}ro summation of functions
$f \in H_1(\mathbb{T}^\infty)$ as well as a  Helson type theorem on logarithmic summation of vertical limits of
ordinary Dirichlet series. From Theorem \ref{maximalineqtypeI} and some standard arguments (which as mentioned above were formalized in
\cite[Lemma 3.6]{DefantSchoolmann3}) the next two results are given in
\cite[Corollary 2.2]{DefantSchoolmann3} and
\cite[Corollary 2.7]{DefantSchoolmann3}.

\begin{Coro}
Let $f \in H_{1}(\T^{\infty})$.
\begin{itemize}
\item[(1)]
For almost all $z\in \T^{\infty}$
\[
\lim_{x\to \infty} \sum_{\log \mathfrak{p}^{\alpha} <x} \widehat{f}(\alpha) \left(1-\frac{\log \mathfrak{p}^{\alpha}}{x}\right)z^{\alpha}=f(z).
\]
\item[(2)]
There is a null set $N \subset \mathbb{T}^\infty$ such that for all $u>0$ and all  $z\notin N $
\begin{align*}
\lim_{x\to \infty} \sum_{\log \mathfrak{p}^{\alpha} <x} \widehat{f}(\alpha)
  e^{-u \log \mathfrak{p}^{\alpha}}
    \left(1-\frac{\log \mathfrak{p}^{\alpha}}{x}\right)z^{\alpha} = f \ast p_u (z)
\end{align*}
\end{itemize}
\end{Coro}

Note that
$f \ast p_u \in H_1(\mathbb{T}^\infty)$
for all $u >0$ and $f \in H_1(\mathbb{T}^\infty)$, and hence by (1) we know that $\lim_{x\to \infty} R_x(f \ast p_u) = f \ast p_u $
almost everywhere. But the point of (2) is that the null set $N$ works  for all $u$
simultaneously. In \cite[Proposition 2.4]{DefantSchoolmann3} it is proved that
$\lim_{u \to \infty} f \ast p_u = f$ almost everywhere, so (1) is definitely the border case of (2).\\

 Lemma \ref{hedesaks*} transports  these results on almost everywhere pointwise limit of Riesz means of functions  on the infinite dimensional torus to  almost everywhere pointwise convergence  on the imaginary axis of almost all the vertical limits of their associated Dirichlet series.

\begin{Coro}
Let   $D \in \mathcal{H}_1$ and $f \in H_{1}(\T^{\infty})$ its associated function under the Bohr transform. Then
there is a null set $N \subset \Xi$ such that for all $\chi \notin N$
\begin{itemize}
\item[(1)]
$
\displaystyle
\lim_{x\to \infty} \sum_{\log n < x} a_n \chi(n)\left(1-\frac{\log n}{x}\right)=
f\big(\chi(\mathfrak{p})\big)
$
\item[(2)]
$
\displaystyle
\lim_{x\to \infty} \sum_{\log n < x} a_n \chi(n)\left(1-\frac{\log n}{x}\right) n^{-it}=
f\Big( \frac{\chi(\mathfrak{p})}{\mathfrak{p}^{it}}\Big)
$
for almost all $t \in \mathbb{R}$
\item[(3)]
$
\displaystyle
\lim_{x\to \infty} \sum_{\log n < x} a_n \chi(n)\left(1-\frac{\log n}{x}\right) n^{-(u+it)}=
f_\chi \ast P_u (t)
$ for  all $u+it \in [\text{Re}>0]$
\end{itemize}
\end{Coro}

 \section{Helson meets Hardy-Littlewood} \label{HaLi}
Here we want to sketch the proof of Theorem \ref{maximalineqtypeI}, since we feel that it has some features  which are  independently interesting.\\

 One of the central tools  is given by the following Hardy-Littlewood maximal operator. If $f \in L_{1}(\mathbb{T}^\infty)$, then we define for $z \in \mathbb{T}^\infty$
\begin{equation} \label{maximaloperatordef}
\overline{M}(f)(z):=\sup_{I\subset \R} \frac{1}{|I|}\int_{I} \Big|f\Big(\frac{z}{\mathfrak{p}^{it}}\Big)\Big| ~dt,
\mathfrak{}\end{equation}
where $I$ stands for any interval in $\R$ with  Lebesgue measure $|I|$. Recall from \eqref{flowflow} that
$f\Big(\frac{z}{\mathfrak{p}^{it}}\Big)= f(\beta(t)z) =  f_{z}(t)$ for almost all $z \in \mathbb{T}^\infty$ defines a locally integrable function on $\R$, and so $\overline{M}(f)(z)$ is defined almost everywhere.

\begin{Theo} \label{HLoperator} The sublinear operator $\overline{M}$ is bounded from $L_{1}(\mathbb{T}^\infty)$ to $L_{1,\infty}(\mathbb{T}^\infty)$ and from $L_{p}(\mathbb{T}^\infty)$ to $L_{p}(\mathbb{T}^\infty)$, whenever $1<p\le \infty.$
\end{Theo}

The details of the proof are given in \cite[Theorem 2.1]{DefantSchoolmann3}, and it is not too surprising that the first part of
the proof uses Vitali's covering lemma, whereas the second part then follows applying the Marcinkiewicz  interpolation theorem.\\

 But before we discuss how to apply the preceding theorem,  let us give the following direct consequence, which is of independent interest (see \cite[Corollary 2.11]{DefantSchoolmann3}, and compare the second equality with  \eqref{norm}).
\begin{Coro} \label{differentiation} Let $f \in L_{1}(\mathbb{T}^\infty)$. Then for almost all $z \in \mathbb{T}^\infty$ we have
\begin{equation*} \label{difff}
\lim_{T\to 0} \frac{1}{2T} \int_{-T}^{T}f\Big(\frac{z}{\mathfrak{p}^{it}}\Big)dt =f(z),
\end{equation*}
\begin{equation*} \label{besiconorm2}
\lim_{T\to \infty} \frac{1}{2T} \int_{-T}^{T} f\Big(\frac{z}{\mathfrak{p}^{it}}\Big)~ dt=
\int_{\mathbb{T}^\infty}  f(w) dw.
\end{equation*}
\end{Coro}

The next theorem reduces the proof of Theorem \ref{maximalineqtypeI} to Theorem \ref{HLoperator}.
\begin{Theo} \label{reduction} Let  $f \in H_{1}(\mathbb{T}^\infty)$. Then  for almost all $w \in \mathbb{T}^\infty$ have
\begin{equation*} \label{reductioneq}
T(f)(w)=\sup_{x>0} \bigg|\sum_{\log \mathfrak{p}^\alpha <x} \widehat{f}(\alpha) \left(1-\frac{\log \mathfrak{p}^\alpha}{x}\right) z^\alpha\bigg|\le C\,\overline{M}(f)(z)\,,
\end{equation*}
where $C>0$ is an  absolute constant.
\end{Theo}

We finish this section indicating  how to prove this inequality. The proof starts with two concrete integrals.
The first integral  follows from a standard  application of Cauchy's theorem (it is a particular  case of \cite[Lemma 10, p.50]{HardyRiesz}),

\begin{equation}\label{genius}
\frac{1}{2\pi i}\int_{c-i\infty}^{c+i\infty} \frac{e^{ys}}{s^{2}} ds = \begin{cases} y& \text{ for } y\ge 0\\ 0 & \text{ for } y<0 \end{cases}
\,\,\, \text{ and $c >0$ }
\end{equation}
 whereas the  second one is a consequence of an elementary  calculation:
\begin{equation} \label{lemma1type1ineq2}
\int_{\R} \frac{P_{v}(y-t)}{v^2 +y^2} dy=\frac{2}{4v^{2}+t^{2}}\,\,\, \text{ for $t,v >0$  }.
\end{equation}
From these two integrals we deduce some sort of Perron formula for logarithmic means.

\begin{Prop} \label{hallo}
Let  $f \in H_1(\mathbb{T}^\infty)$. Then there is a null set $N$ in $\mathbb{T}^\infty$
such that for all $z \notin N$ and for all $x >0$
\begin{equation*} \label{perroneq}
\frac{2\pi x }{e} R_{x}(f)(z)= \int_{\R} f_{z}(a) \mathcal{F}_{L_{1}(\R)} \left( \frac{P_{\frac{1}{x}}(\cdot-a)}{(\frac{1}{x}+i\cdot)^{2}} \right)(-x) da\,,
\end{equation*}
where $\mathcal{F}_{L_{1}(\R)}$ stands for the Fourier transform on $L_{1}(\R)$.
\end{Prop}
\begin{proof}
Let us first assume that $f =\sum_{\mathfrak{p}^\alpha < N} \hat{f}(\alpha) z^\alpha \in H_1(\mathbb{T}^\infty)$. Then by \eqref{genius} for $c >0$ and $z \in \mathbb{T}^\infty$
\begin{align*}
e^{xc}\int_{\R}   f_{z}(a) \int_{\R}  \frac{P_{c}(t-a)}{(c+it)^{2}}e^{ixt} dt ~ da
&
= e^{xc}\int_{\R} \frac{(f_{z}*P_{c})(t)}{(c+it)^{2}} e^{ixt} dt
\\&
=\int_{\R} \frac{\sum_{\log \mathfrak{p}^\alpha < N} \hat{f}(\alpha) z^\alpha
e^{-\log \mathfrak{p}^\alpha (c+it)}}{(c+it)^{2}} e^{x(c+it)} dt
\\&
=
2\pi x\sum_{\log \mathfrak{p}^\alpha < x} \hat{f}(\alpha) z^\alpha
\left(1-\frac{\log \mathfrak{p}^\alpha}{x}\right)\,,
\end{align*}
and the  choice $c= 1/x$ leads to the conclusion.
To prove this for arbitrary $f \in H_1(\mathbb{T}^\infty)$, observe that for all  $v>0$ the operator
\begin{equation} \label{operatorA}
\mathcal{A} \colon L_{1}(\mathbb{T}^\infty) \to L_{1}(\mathbb{T}^\infty,L_{1}(\R)), ~~ f \mapsto \left[z \to \frac{f_{z}*P_{v}}{\left(v+i\cdot\right)^{2}}\right]
\end{equation}
is bounded. Indeed, by \eqref{lemma1type1ineq2} and Fubini's theorem  for every $f \in L_{1}(\mathbb{T}^\infty)$
\begin{align*}
\|\mathcal{A}(f)\|_1
&
=\int_{\mathbb{T}^\infty} \left| \int_{\R} \frac{f_{z}*P_{v}(y)}{\left(v+iy\right)^{2}} dy \right| dz
\\
&
\leq \int_{\mathbb{T}^\infty} \int_{\R} |f_{z}(t)| \int_{\R} \frac{P_{v}(y-t)}{v^2 + y^2} dy dt dz
\\
&
\le \int_{\mathbb{T}^\infty} \int_{\R} |f_{z}(t)| \frac{2}{4v^{2}+t^{2}} dt dz
= C_{1}(v)\|f\|_{1}.
\end{align*}
Additionally, this shows, that for some null set $N$ in $\mathbb{T}^\infty$
we have that $\frac{f_{z}*P_{v}}{(v+i\cdot)^{2}} \in L_{1}(\R)$ for all $z \notin N$, and so we in particular deduce that for $z \notin N$ and $x>0$
\begin{equation} \label{Fourier2}
\mathcal{F}_{L_{1}(\R)}\left( \frac{f_{z}*P_{\frac{1}{x}}}{\left(\frac{1}{x}+i\cdot\right)^{2}}\right)(-x)=\int_{\R} f_{z}(a) \mathcal{F}_{L_{1}(\R)} \left( \frac{P_{\frac{1}{x}}(\cdot-a)}{(\frac{1}{x}+i\cdot)^{2}} \right)(-x) da.
\end{equation}
Now let $(Q^{n})$ be a sequence of polynomials in $H_1(\mathbb{T}^\infty)$ converging to $f$ in $H_1(\mathbb{T}^\infty)$ (see e.g. \cite[Proposition 5.5]{Defant}). Then, by the continuity of $\mathcal{A}$ and $\mathcal{F}_{L_{1}(\R)}$, we for all $z \notin N$ obtain some
subsequence $(Q^{n_{k}})$ such that under uniform convergence on $\R$
$$\mathcal{F}_{L_{1}(\R)}\left( \frac{f_{z}*P_{\frac{1}{x}}}{\left(\frac{1}{x}+i\cdot\right)^{2}}\right)=\lim_{k\to \infty}\mathcal{F}_{L_{1}(\R)}\left( \frac{Q_{z}^{n_{k}}*P_{\frac{1}{x}}}{\left(\frac{1}{x}+i\cdot\right)^{2}}\right).$$
 So, knowing that the claim holds true for polynomials, by \eqref{Fourier2} for all  $z \notin N$ and $x>0$
\begin{align*}
\mathcal{F}_{L_{1}(\R)}\left( \frac{f_{z}*P_{\frac{1}{x}}}{\left(\frac{1}{x}+i\cdot\right)^{2}}\right)(-x)
&
=\lim_{k\to \infty}\mathcal{F}_{L_{1}(\R)}\left( \frac{Q_{z}^{n_{k}}*P_{\frac{1}{x}}}{\left(\frac{1}{x}+i\cdot\right)^{2}}\right)(-x)
\\ &=\frac{2\pi x }{e}\lim_{k\to \infty} R_{x}(Q^{n_{k}})(z)=\frac{2\pi x }{e}R_{x}(f)(z),
\end{align*}
which  looking again at \eqref{Fourier2} finishes the argument.
\end{proof}

Finally, we may combine all this to get the

\begin{proof}[Proof of Theorem \ref{reduction}]
For $a \in \mathbb{R}$ and $x>0$ we define
\[
K(a):=\frac{1}{|1+ia|^{2}}\,\,\, \text{and} \,\,\,K_{x}(a):=x K(ax)=\frac{x}{|1+iax|^{2}}\,.
\]
 Then with Proposition \ref{hallo} and  \eqref{lemma1type1ineq2} we obtain for almost all $z \in \mathbb{T}^\infty$ and $x>0$
\begin{align*}
|R_{x}(f)(z)|
&
\le \frac{C_{1}}{ x} \int_{\R} |f_{z}(a)| \left\|\frac{P_{\frac{1}{x}}(\cdot-a)}{(\frac{1}{x}+i\cdot)^{2}}\right\|_{L_{1}(\R)} da
\\ &
\le\frac{C_{2}}{x} \int_{\R}|f_{z}(a)| \frac{1}{|\frac{1}{x}+ia|^{2}} ~da
\\ &
= C_{3} \int_{\R}|f_{z}(a)| K_{x}(a) ~da =C_{3} (|f_{z}|*K_{x})(0).
\end{align*}
Now by \cite[Theorem 2.1.10, p.91]{Grafakos1} we have for almost all $z \in \mathbb{T}^\infty$
$$\sup_{x>0} |f_{z}|*K_{x}(0)\le \|K\|_{L_{1}(\R)} \sup_{T>0} \frac{1}{2T} \int_{-T}^{T} |f_{z}(t)| dt \le \|K\|_{L_{1}(\R)} \overline{M}(f)(z),$$
which then  all in all gives the conclusion.
\end{proof}

    \section{Helson meets Banach}
    A sequence $(x_n)$ in a Banach space $X$ is a Schauder basis whenever every $x \in X$ has a unique
    series representation $x = \sum_{n=1}^\infty \alpha_n x_n$.
    For $\mathcal{H}_p$'s of Dirichlet series the following result was first realized in \cite{Saksman}.

    \begin{Theo} \label{zero}
    Let $1 < p < \infty$. Then
         for every  $D= \sum a_{n}n^{-s}\in \mathcal{H}_p$
    $$
    D = \lim_{x \to \infty}\sum_{n=1}^x a_{n}n^{-s}
    $$
    converges in $\mathcal{H}_p$, i.e. the Dirichlet series $n^{-s}, n \in \mathbb{N}$, form a Schauder basis of $\mathcal{H}_p$. Equivalently,
    for every  $f \in H_p(\mathbb{T}^\infty)$
    $$
    f=  \lim_{x \to \infty }\sum_{ \mathfrak{p}^\alpha < x }\hat{f}(\alpha) z^\alpha
    $$
    converges in $H_p(\mathbb{T}^\infty)$. In other words, the monomials $z^\alpha$ ordered
    as in \eqref{order} (i.e. $z^\alpha \leq z^\beta :\Leftrightarrow
    \log \mathfrak{p}^\alpha \leq\log \mathfrak{p}^\beta$ ) form a Schauder basis of $H_p(\mathbb{T}^\infty)$.
          \end{Theo}

An equivalent formulation of  all this   is that
for $1<p<\infty$
all projections
\begin{equation} \label{Riesz}
S_x^p: \mathcal{H}_p \to \mathcal{H}_p,\,\,
\sum a_n n^{-s} \mapsto \sum_{n<x} a_n n^{-s}
\end{equation}
 are uniformly bounded, and this is an immediate consequence of Theorem \ref{one}.
  As explained in Section \ref{fail} the proof of Theorem \ref{one} is based on the deep Carleson-Hunt Theorem
  \ref{CarlesonHunt}.\\

  But  let us remark that Theorem \ref{zero} is also
  an almost  straightforward consequence of
      Rudin's work  from \cite[Theorem 8.7.2]{Rudin62}.   There Rudin proves (as a particular case of a more general result on
      compact abelian groups with ordered duals, see \eqref{order}  for the order on $\widehat{\T^\infty}$) that   the so-called Riesz projection
      \[
      L_2(\mathbb{T}^\infty) \to L_2(\mathbb{T}^\infty), \, \,\,
      f \mapsto \sum_{\log \mathfrak{p}^\alpha \ge 0} \hat{f}(\alpha) z^\alpha
      \]
     for every $1 < p < \infty$ extends to a bounded operator on $L_p(\mathbb{T}^\infty)$.
For  the border cases  $p=\infty$ and $p=1$ this is false.
      Indeed, from  Theorem \ref{Bohr} one may deduce that
      \[
 \log \log x \ll \|S_x^\infty: \mathcal{H}_\infty \to \mathcal{H}_\infty \| \ll \log x\,,
 \]
whereas
 a very recent estimate  from \cite{BoBrSaSe} states that
 \[
 \log \log x \ll \|S_x^1: \mathcal{H}_1 \to \mathcal{H}_1 \| \ll \frac{\log x}{\log  \log x}
 \]
(in both cases the lower estimates are simple consequences of the one variable case and
Bohr's transform from \eqref{Bohrtrafo}.\\

Finally, we ask a question for the Banach spaces $\mathcal{H}_1$ and $H_1(\mathbb{T}^\infty)$
which is analog to the one we posed in Section \ref{fail}.\\

    \noindent {\bf Question:} {\it Is it true that for all $D \in \mathcal{H}_1$ we have that
    \[
    \lim_{N \to \infty}  \frac{1}{N} \sum_{k=0}^{N-1} \sum_{n\le k} a_n n^{-s} =D \,\,\,\, \text{in $\mathcal{H}_1$\,,}
    \]
 or, equivalently, do we for $f \in H_1(\mathbb{T}^\infty)$ have that
\begin{equation} \label{labalaba}
    \lim_{N \to \infty}  \frac{1}{N} \sum_{k=0}^{N-1}\sum_{\mathfrak{p^\alpha} \le  k} \hat{f}(\alpha) z^\alpha=f \,\,\,\, \text{in $H_1(\mathbb{T}^\infty)$?}
    \end{equation}
    }

The answer is again no. Otherwise  the same argument as in Section \ref{fail}  would show that the Fourier  series of every
$f \in H_1(\mathbb{T})$ converges in $H_1(\mathbb{T})$ which we know is false.

\end{document}